\documentclass[11pt]{amsart}
\usepackage{amscd,amssymb}
\usepackage{amsthm,amsmath}

\usepackage{fancyhdr}
\usepackage{supertabular}
\usepackage{setspace}
\usepackage{color}
\usepackage{longtable}
\usepackage{epic}
\usepackage{curves}
\usepackage[matrix,arrow,curve]{xy}
\usepackage{graphicx}
\usepackage{enumerate}
\usepackage{amsfonts}
\usepackage{cite}
\usepackage{pifont}
\usepackage{multirow}
\usepackage{hhline}
\usepackage[dvipsnames]{xcolor}

\usepackage{pgf,pgfarrows,pgfnodes,pgfautomata,pgfheaps,pgfshade}
\usepackage{tikz}
\newcommand{\blowup}[1]{\tikz \node[draw,circle, inner sep=0pt, minimum size=.5mm]{#1};}

\usepackage{lmodern}

\makeatletter
\ifcase \@ptsize \relax
  \newcommand{\miniscule}{\@setfontsize\miniscule{3}{7}}
    \newcommand{\stiny}{\@setfontsize\miniscule{5}{7}}
\or
  \newcommand{\miniscule}{\@setfontsize\miniscule{3}{7}}
   \newcommand{\stiny}{\@setfontsize\miniscule{5}{7}}
\or
  \newcommand{\miniscule}{\@setfontsize\miniscule{3}{7}}
    \newcommand{\stiny}{\@setfontsize\miniscule{5}{7}}
\fi
\makeatother

\newtheorem{theorem}[equation]{Theorem}
\newtheorem{lemma}[equation]{Lemma}

\newtheorem{corollary}[equation]{Corollary}

\newtheorem*{question*}{Question}

\newtheorem*{problem*}{Problem}

\theoremstyle{definition}
\newtheorem{example}[equation]{Example}
\newtheorem{definition}[equation]{Definition}
\newtheorem{remark}[equation]{Remark}

\theoremstyle{remark}

\makeatletter\@addtoreset{equation}{section} \makeatother

\newcommand{\one}{\mbox{\tiny \ding{172}}}
\newcommand{\two}{\mbox{\tiny \ding{173}}}
\newcommand{\three}{\mbox{\tiny \ding{174}}}
\newcommand{\four}{\mbox{\tiny \ding{175}}}
\newcommand{\five}{\mbox{\tiny \ding{176}}}
\newcommand{\six}{\mbox{\tiny \ding{177}}}
\newcommand{\seven}{\mbox{\tiny \ding{178}}}
\newcommand{\eight}{\mbox{\tiny \ding{179}}}
\newcommand{\nine}{\mbox{\tiny \ding{180}}}
\newcommand{\ten}{\mbox{\tiny \ding{181}}}

\newcommand{\eleven}{\mbox{\miniscule \blowup{\textbf{11}}}}
\newcommand{\twelve }{\mbox{\miniscule \blowup{\textbf{12}}}}
\newcommand{\thirteen}{\mbox{\miniscule \blowup{\textbf{13}}}}

\newcommand{\bone}{\mbox{\tiny \ding{182}}}

\newcommand{\bthree}{\mbox{\tiny \ding{184}}}
\newcommand{\bfour}{\mbox{\tiny \ding{185}}}
\newcommand{\bfive}{\mbox{\tiny \ding{186}}}
\newcommand{\bsix}{\mbox{\tiny \ding{187}}}

\title{Cylinders in singular del Pezzo surfaces}

\author{Ivan Cheltsov, Jihun Park and Joonyeong Won}

\address{ \emph{Ivan Cheltsov}\newline \textnormal{School of Mathematics, The
University of Edinburgh
\newline \medskip James Clerk Maxwell Building,
The King's Buildings,
Mayfield Road, Edinburgh EH9 3JZ, UK.\newline
Laboratory of Algebraic Geometry, National Research University Higher School of Economics
\newline 7 Vavilova Str. Moscow, 117312, Russia.\newline
 \texttt{I.Cheltsov@ed.ac.uk}}}

\address{ \emph{Jihun Park}\newline \textnormal{Center for Geometry and Physics, Institute for Basic Science (IBS)
\newline \medskip 77 Cheongam-ro, Nam-gu, Pohang, Gyeongbuk, 37673, Korea. \newline
Department of
Mathematics, POSTECH \newline
 77 Cheongam-ro, Nam-gu, Pohang, Gyeongbuk,  37673, Korea. \newline
\texttt{wlog@postech.ac.kr}}}

\address{\emph{Joonyeong
Won}\newline \textnormal{KIAS\newline 85 Hoegiro Dongdaemun-gu, Seoul
130-722, Korea.\newline \texttt{leonwon@kias.re.kr}}}

\begin{document}

\begin{abstract}
For each del Pezzo surface $S$ with du Val singularities, we determine whether it admits a $(-K_S)$-polar cylinder or not. If it allows one, then we present an effective $\mathbb{Q}$-divisor $D$ that is $\mathbb{Q}$-linearly  equivalent to $-K_S$ and such that the open set $S\setminus\mathrm{Supp}(D)$ is a cylinder. As a corollary, we classify all the del Pezzo surfaces with du Val singularities that admit nontrivial  $\mathbb{G}_a$-actions on their affine cones defined by their anticanonical divisors.
\end{abstract}

\maketitle

All considered varieties are assumed to be
algebraic and defined over an algebraically closed field of
characteristic $0$ throughout this article.

\section{Introduction}
\label{sec:intro}

Let $X$ be a projective variety and $H$ be an ample
divisor on $X$. The \emph{generalised cone} over the polarised variety
$(X,H)$ is the  affine variety defined by
$$
\hat{X}=\mathrm{Spec}\left(\bigoplus_{n\geqslant 0}H^0\left(X, \mathcal{O}_{X}\left(nH\right)\right)\right).%
$$
The affine variety $\hat{X}$ is the usual cone
over the embedded image of $X$ in a projective space by the linear system $|H|$
if $H$ is very ample and the image of the variety  $X$ is projectively normal.

The question of whether the generalised cone  of a given polarised variety $(X, H)$ admits a nontrivial $\mathbb{G}_a$-action has been studied extensively in  \cite{CheltsovParkWon}, \cite{KPZ11a}, \cite{KPZ11b},
\cite{KPZ12a} and
\cite{KPZ12b}. 
The present article is focused on singular del Pezzo surfaces $S_d$ polarised by anticanonical divisors $-K_{S_d}$ to extend the results in  \cite{CheltsovParkWon},  \cite{KPZ11a}
and
\cite{KPZ12b} to the singular del Pezzo surfaces.  Indeed, it classifies all the del Pezzo surfaces with du Val singularities that admit nontrivial  $\mathbb{G}_a$-actions on their generalised cones
over $(S_d, -K_{S_d})$.

Let $S_d$ be a del Pezzo surface of degree $d$ with at worst du Val singularities
and let $\hat{S}_d$ be the generalised cone  over $(S_d, -K_{S_d})$.
For $3\leqslant d\leqslant 9$, the anticanonical divisor  is very ample and the anticanonical linear system embeds $S_d$ into the projective space $\mathbb{P}^d$. The embedded surface $S_d\subset \mathbb{P}^d$ is projectively normal.
Therefore the generalised cone $\hat{S}_d$ is the affine cone in $\mathbb{A}^{d+1}$ over  the  variety embedded in $\mathbb{P}^d$. In particular, for $d=3$, the  surface $S_3$  anticanonically embedded in $\mathbb{P}^3$ is defined by a cubic homogenous polynomial equation, and hence the generalised cone $\hat{S}_3$ is the affine hypersurface in $\mathbb{A}^4$ defined by the same cubic polynomial equation.
Meanwhile, for $d=2$ (resp. $d=1$), the generalised cone $\hat{S}_d$ is the affine cone in $\mathbb{A}^{4}$ over  the hypersurface  in the weighted projective space
$\mathbb{P}(1,1,1,2)$  (resp. $\mathbb{P}(1,1,2,3)$) defined by a quasi-homogeneous polynomial of degree $4$ (resp. $6$) (\cite[Theorem~4.4]{HiWa81}).

The group of T.~Kishimoto, Yu.~Prokhorov, M.~Zaidenberg and the group of I.~Cheltsov, J.~Park, J.~Won  have studied  existence of
nontrivial $\mathbb{G}_a$-actions on such affine cones and obtained results for smooth del Pezzo surfaces.
\begin{theorem}\label{theorem:KPZ-actions-4-9}
For a smooth del Pezzo surface $S_d$ of degree $4\leqslant d\leqslant 9$,  its generalised cone $\hat{S}_d$ admits an effective $\mathbb{G}_a$-action.
\end{theorem}
\begin{proof}
See \cite[Theorem~3.19]{KPZ11a}.
\end{proof}
\begin{theorem}\label{theorem:KPZ-actions-1-2-3}
For a smooth del Pezzo surface $S_d$ of degree $d\leqslant 3$,  its generalised cone $\hat{S}_d$ admits  no nontrivial $\mathbb{G}_a$-action.
\end{theorem}
\begin{proof}
See \cite[Theorem~1.1]{KPZ12b}
and \cite[Corollary~1.8]{CheltsovParkWon}.
\end{proof}

Their proofs make good use of a geometric property called cylindricity, which is worthwhile to study for its own sake.
\begin{definition}
\label{definition:polar-cylinder} Let $M$  be a $\mathbb{Q}$-divisor on a  projective normal variety  $X$.  An $M$-polar cylinder in $X$ is
an open subset
$$
U=X\setminus \mathrm{Supp}(D)
$$
defined by an effective
$\mathbb{Q}$-divisor $D$  in the $\mathbb{Q}$-linear equivalence class of $M$   such that $U$ is isomorphic to $Z\times\mathbb{A}^1$ for some affine variety $Z$.
\end{definition}
It is shown that   the existence of a nontrivial $\mathbb{G}_a$-action on the generalised cone over $(X, H)$ is equivalent to the existence of an $H$-polar cylinder on $X$.

\begin{lemma}\label{theorem:KPZ-criterion} Let $X$ be a projective normal variety equipped with  an ample Cartier
divisor  $H$ on $X$. Suppose that the generalised  cone $\hat{X}$ over $(X,H)$ is normal. Then $\hat{X}$
admits an effective $\mathbb{G}_{a}$-action if and only if $X$
contains an $H$-polar cylinder.
\end{lemma}
\begin{proof}
See \cite[Theorem~2.1]{KPZ12b}.
\end{proof}
Indeed, in order to prove Theorems ~\ref{theorem:KPZ-actions-4-9} and \ref{theorem:KPZ-actions-1-2-3}, they show that  a smooth del Pezzo surface $S_d$ has
a $(-K_{S_d})$-polar cylinder if $4\leqslant d\leqslant 9$ but  no  $(-K_{S_d})$-polar cylinder if $d\leqslant 3$.

The goal of the present article is to extend the results of  \cite{CheltsovParkWon}, \cite{KPZ11a}  and \cite{KPZ12b} to del Pezzo surfaces with du Val singularities. To be precise, we prove the following:

\begin{theorem}[cf. Remark~\ref{remark:tiger}]
\label{theorem:main} Let $S_d$ be a del Pezzo surface of degree $d$ with at most
du Val singularities.
\begin{itemize}
\item[I.] The surface $S_d$ does not admit a $(-K_{S_d})$-polar cylinder when
\begin{enumerate}
\item $d=1$ and  $S_d$ allows only singular points of types $\mathrm{A}_1$, $\mathrm{A}_2$, $\mathrm{A}_3$,
$\mathrm{D}_4$ if any;
\item $d=2$ and $S_d$ allows only singular points of type $\mathrm{A}_1$ if any;
\item $d=3$ and $S_d$ allows no singular point.%
\end{enumerate}
\item[II.] The surface $S_d$ has a $(-K_{S_d})$-polar cylinder if it is not one of the del Pezzo surfaces listed in I.
\end{itemize}
\end{theorem}
Theorem~\ref{theorem:main} immediately implies the following via Lemma~\ref{theorem:KPZ-criterion}
\begin{corollary}
\label{corollary:main}
Let $S_d$ be a del Pezzo surface of degree $d$ with at most
du Val singularities. Then the affine cone over $(S_d, -K_{S_d})$ does not admit a nontrivial $\mathbb{G}_a$-action exactly when
\begin{enumerate}
\item $d=1$ and  $S_d$ allows only singular points of types $\mathrm{A}_1$, $\mathrm{A}_2$, $\mathrm{A}_3$,
$\mathrm{D}_4$ if any;
\item $d=2$ and $S_d$ allows only singular points of type $\mathrm{A}_1$ if any;
\item $d=3$ and $S_d$ allows no singular point,%
\end{enumerate}

\end{corollary}

As mentioned before, the theorem has been verified for smooth del Pezzo surfaces  in
\cite{CheltsovParkWon}, \cite{KPZ11a} and \cite{KPZ12b}. For this reason, from now on, we consider only singular del Pezzo surfaces.
The cone over an irreducible conic curve in $\mathbb{P}^3$ obviously has a cylinder. Indeed,  a quadruple ruling line is $\mathbb{Q}$-linearly equivalent to the anticanonical class and its complement is isomorphic to $\mathbb{A}^2$. Therefore, we may exclude this singular del Pezzo surface from our consideration.

\section{Preliminaries}%
\label{section:preliminaries}

\subsection{Singularities and Inequalities}

Let $S$ be a~projective surface with at most du Val singularities.
In addition,  let $D$ be an
effective $\mathbb{Q}$-divisor on $S$.

\begin{lemma}
\label{lemma:mult} If  the log pair $(S,D)$ is not log canonical at  a smooth point $P$, then
$\mathrm{mult}_{P}(D)>1$.
\end{lemma}
\begin{proof}
See \cite[Proposition~9.5.13]{La04II}.
\end{proof}

Write $D=\sum_{i=1}^{r}a_{i}D_{i}$, where $D_{i}$'s are distinct
prime divisors  and $a_{i}$'s are positive rational numbers.
\begin{lemma}
\label{lemma:convexity} Let $T$ be an effective
$\mathbb{Q}$-divisor on $S$ other than the divisor $D$  such that 
$T\sim_{\mathbb{Q}} D$  and $\mathrm{Supp}(T)\subset \mathrm{Supp}(D)$. For every non-negative
rational number $\epsilon$, put $D_\epsilon=(1+\epsilon)D-\epsilon
T$. Then
\begin{enumerate}
\item $D_{\epsilon}\sim_{\mathbb{Q}} D$ for every
$\epsilon$; \item the set $
\left\{\epsilon\in\mathbb{Q}_{> 0}\ \vert\ D_\epsilon\ \text{is
effective}\right\} $ attains the maximum $\mu$;
 \item at least one component of the support of the divisor
$T$ is not contained in  the support
of the divisor $D_\mu$; \item if  the log pair $(S,T)$ is log canonical at a point $P$ but
$(S,D)$ is not log canonical at $P$, then the log pair $(S,D_{\mu})$ is not log
canonical at $P$.
\end{enumerate}
\end{lemma}

\begin{proof}
See \cite[Lemma~2.2]{CheltsovParkWon}.
\end{proof}

The following is a ready-made adjunction for our situation. 
\begin{lemma}
\label{lemma:adjunction} Suppose that the log pair $(S,D)$ is not
log canonical at a smooth point $P$. If a component $D_j$ with $a_j\leqslant 1$
is smooth at  $P$, then
$$
D_{j}\cdot\left(\sum_{i\ne j}a_{i}D_{i}\right)\geqslant\sum_{i\ne j }a_{i}\left(D_{i}\cdot D_{j}\right)_P>1,%
$$
where $\left(D_{i}\cdot D_{j}\right)_P$ is the local intersection number of $D_i$ and $D_j$ at $P$.
\end{lemma}
\begin{proof}
See
\cite[Theorem 5.50]{KoMo}.
\end{proof}

The following is an easy application of Lemma~\ref{lemma:adjunction}.

\begin{lemma}
\label{lemma:D4}
Suppose that the surface $S$ has a singular point $P$  of type $\mathrm{D}_4$.
Let
$g\colon\tilde{S}\to S$ be the minimal resolution of the point $P$.
Denote by $E_1$, $E_2$, $E_3$ and $E_4$ the $g$-exceptional curves, where $E_3$ is the $(-2)$-curve intersecting the other three $(-2)$-curves.
Write
$$
\tilde{D}=g^{*}(D)-\sum_{i=1}^{4}a_iE_i,
$$
where $\tilde{D}$ is the proper transform of $D$ by $g$.
Then the log pair $(S,D)$ is not log canonical at  $P$ if and only if $a_3>1$.
\end{lemma}
\begin{proof}
See \cite[Lemma~2.5]{ChK10}.
\end{proof}

\begin{remark}
\label{remark:log-pull-back}
Let $f\colon \bar{S}\to S$ be the blow up of  the surface $S$ at
a smooth point $P$ with the exceptional divisor $E$ and let $\bar{D}$
be the proper transform of $D$ by $f$. Then we have
$$
K_{\bar{S}}+\bar{D}+\left(\mathrm{mult}_{P}(D)-1\right)E=f^{*}\left(K_{S}+D\right).
$$
The log pair $(S,D)$ is log canonical at  $P$ if and only if the
log pair $(\bar{S}, \bar{D}+(\mathrm{mult}_{P}(D)-1)E)$ is log
canonical along  $E$.
 If $(S,D)$ is not log
canonical at  $P$, then there exists a point $Q$ on  $E$ at which
 $(\bar{S}, \bar{D}+(\mathrm{mult}_{P}(D)-1)E)$ is
not log canonical. Lemma~\ref{lemma:mult} then implies
\begin{equation}
\mathrm{mult}_{P}(D)+\mathrm{mult}_{Q}(\bar{D})>2.
\end{equation}
 If $\mathrm{mult}_{P}(D)\leqslant 2$,
then  $(\bar{S},
\bar{D}+(\mathrm{mult}_{P}(D)-1)E)$ is log canonical at every
point on $E$ except the point $Q$. Indeed, if
it is not log canonical at
another point $O$ on $E$, then Lemma~\ref{lemma:adjunction}
yields  a contradiction, 
$$
2\geqslant\mathrm{mult}_{P}(D)=\bar{D}\cdot E\geqslant\mathrm{mult}_{Q}(\bar{D})+\mathrm{mult}_{O}(\bar{D})>2.%
$$
\end{remark}
The following lemma will be useful for the article.
\begin{lemma}
\label{lemma:Trento}  Let $C_1$ and $C_2$ be irreducible curves on the surface $S$ that both are
smooth at a smooth point $P$ and intersect transversally at $P$. In addition, let $\Omega$ be an effective
$\mathbb{Q}$-divisor on  $S$ whose support
contains neither $C_1$ nor  $C_2$.
Suppose that the log pair $(S, a_{1}C_{1}+a_{2}C_{2}+\Omega)$ is not log canonical at  $P$
 for some non-negative rational numbers $a_1$, $a_2$.
 If $\mathrm{mult}_{P}(\Omega)\leqslant 1$, then either
$$\mathrm{mult}_{P}(\Omega\cdot C_{1})>2(1-a_{2})\ \ \  \mbox{ or } \ \ \
\mathrm{mult}_{P}(\Omega\cdot C_{2})>2(1-a_{1}).$$
\end{lemma}
\begin{proof}
See \cite[Theorem~13]{Ch13}.
\end{proof}

From now on, on a projective surface, an effective  $\mathbb{Q}$-divisor $\mathbb{Q}$-linearly equivalent to the anticanonical class of the surface will be called \emph{an effective anticanonical $\mathbb{Q}$-divisor} and a member of the anticanonical linear system will be called \emph{an effective anticanonical divisor}.

\subsection{Singularity types}

For singular del Pezzo surfaces $S$ of degrees $1$ and $2$ not listed in Theorem~\ref{theorem:main} (I),  the construction method of their  $(-K_S)$-polar cylinders will be given according to the singularity types of the surfaces $S$.  For this purpose, we adopt the following definition.

\begin{definition} Let $S_1$ and $S_2$ be del Pezzo surfaces with at most du Val singularities and let $\tilde{S}_1$ and $\tilde{S}_2$ be their minimal resolutions, respectively.
We say that the del Pezzo surfaces $S_1$ and $S_2$ have the same singularity type (or the minimal resolutions $\tilde{S}_1$ and
$\tilde{S}_2$ have the same type) if there is an isomorphism of the
Picard group of $\tilde{S}_1$ to the
Picard group of   $\tilde{S}_2$ preserving the intersection form that gives a bijection between their sets of classes of negative curves.
\end{definition}

Note that the minimal resolutions of  del Pezzo surfaces with  du Val singularities are smooth weak del Pezzo surfaces, i.e., smooth projective surfaces with nef and big anticanonical divisors. 

It is known that the types of smooth weak del Pezzo surfaces of degree $d$ are in one-to-one correspondence  to
the subsystems of the root systems of types $\textnormal{E}_8$, $\textnormal{E}_7$, $\textnormal{E}_6$, $\textnormal{D}_5$, $\textnormal{A}_4$, $\textnormal{A}_2+\textnormal{A}_1$, $\textnormal{A}_1$,
respectively, for $d=1, \ldots, 7,$ with four exceptions: $8\textnormal{A}_1$,
$7\textnormal{A}_1$, $\textnormal{D}_4+4\textnormal{A}_1$ for $d=1$ and $7\textnormal{A}_1$ for $d=2$
(see \cite{AN06}, \cite{BW79}, \cite{Co88}, \cite{De13}, \cite{Pi77},
\cite{U83}).

 Since  the isomorphisms of the Picard groups of the weak del Pezzo surfaces of the same type  preserve the intersection forms, we can conclude
 from \cite[Th\'eor\`eme~III.2 and Corollaire]{Dema80} that a given singularity
 type has a unique configuration of $(-1)$-curves and $(-2)$-curves. The type of smooth weak del Pezzo surface is uniquely determined by its degree and its  configuration
 of $(-1)$-curves and $(-2)$-curves.
 Consequently, for a
 given singularity type of del Pezzo surfaces of degree $d$, if we find one weak del Pezzo surface of degree $d$ whose
 corresponding singular del Pezzo surface has the given
 singularity type, then this weak del Pezzo surface gives us the configuration
 of $(-1)$-curves and $(-2)$-curves for  the given singularity type
 since every del Pezzo surface of the same
 singularity type has the same configuration of $(-1)$-curves and
 $(-2)$-curves on its weak del Pezzo surface.

 For the singularity types of del Pezzo surfaces of degrees $\leqslant 2$ with at most du Val singularities, 
 we refer to the table in  \cite{U83}.
 The table completely classifies subsystems of the root systems $\textnormal{E}_7$ and $\textnormal{E}_8$
 up to actions of their Weyl groups.

On a del Pezzo surface of a given degree $d$, the configuration of the $(-2)$-curves on the corresponding smooth weak del Pezzo surface does not determine the type uniquely.
In such a case, there are precisely two types.
The following are the ADE-types (with the degrees $d\leq 2$) that have two different singularity types. \bigskip

\begin{minipage}[m]{.95\linewidth}
\begin{itemize}
 \item[$d=1.$] $\textnormal{A}_7$, $\textnormal{A}_5+\textnormal{A}_1$, $2\textnormal{A}_3$,
$\textnormal{A}_3+2\textnormal{A}_1$, $4\textnormal{A}_1$;
\item[$d=2.$] $\textnormal{A}_5+\textnormal{A}_1$, $\textnormal{A}_5$, $\textnormal{A}_3+2\textnormal{A}_1$
$\textnormal{A}_3+\textnormal{A}_1$, $4\textnormal{A}_1$, $3\textnormal{A}_1$;
\end{itemize}
\end{minipage}
\bigskip\\
We need
to distinguish these singularity types of del Pezzo surfaces of degrees $d\leq 2$ with the same ADE-types. However, we do not have to consider the ADE-types $2\textnormal{A}_3$,
$\textnormal{A}_3+2\textnormal{A}_1$, $4\textnormal{A}_1$ for $d=1$ and $4\textnormal{A}_1$, $3\textnormal{A}_1$ for $d=2$
due to Theorem~\ref{theorem:main}~(I).
For the remaining ADE-types, the vertex $v$ in the Dynkin diagram of $\textnormal{A}_{2n+1}$, $n\geqslant 1$,  is called the central vertex if $\textnormal{A}_{2n+1}-v=2\textnormal{A}_n$.

For  the ADE-types
$\textnormal{A}_7$ ($d=1$),
$\textnormal{A}_5$ ($d=2$), $\textnormal{A}_5+\textnormal{A}_1$ ($d=2$), $\textnormal{A}_3$ ($d=4$),
we use
$(\textnormal{A}_7)'$, $(\textnormal{A}_5)'$,  $(\textnormal{A}_5+\textnormal{A}_1)'$ and $(\textnormal{A}_3)'$ if there are $(-1)$-curves intersecting the $(-2)$-curves corresponding to
the central vertices and we use  $(\textnormal{A}_7)''$, $(\textnormal{A}_5)''$,  $(\textnormal{A}_5+\textnormal{A}_1)''$ and $(\textnormal{A}_3)''$ if there are not such $(-1)$-curves.

For  the ADE-types
$\textnormal{A}_5+\textnormal{A}_1$ ($d=1$),  $\textnormal{A}_3+\textnormal{A}_1$ ($d=2$),
$\textnormal{A}_3+2\textnormal{A}_1$  ($d=2$),
we use
$(\textnormal{A}_5+\textnormal{A}_1)'$,  $(\textnormal{A}_3+\textnormal{A}_1)'$,
$(\textnormal{A}_3+2\textnormal{A}_1)'$  if there are $(-1)$-curves intersecting the $(-2)$-curves corresponding to the central vertices and the vertices of $\textnormal{A}_1$  and we use
$(\textnormal{A}_5+\textnormal{A}_1)''$, $(\textnormal{A}_3+\textnormal{A}_1)''$,
$(\textnormal{A}_3+2\textnormal{A}_1)''$  if there are not such $(-1)$-curves.

\section{Absence of Cylinders}
\subsection{Del Pezzo surfaces of degree $1$}
\label{sec:dP1}

Let $S$ be a del Pezzo surface of degree $1$ with at most du
Val singularities. Then its anticanonical linear system $|-K_{S}|$ is a pencil that has a unique
base point. Denote its base point by $O$. Note that the base point $O$ must be a smooth point of the surface $S$.

\begin{theorem}
\label{theorem:dP1}  
Let $D$ be an effective  anticanonical $\mathbb{Q}$-divisor on $S$. 
\begin{enumerate}
\item The log pair $(S,D)$ is log canonical outside of finitely many points.
\item It is log canonical at the point $O$.
\end{enumerate}
Let $P$ be either  a smooth point different from $O$ or  a singular point of type $\mathrm{A}_1$,  $\mathrm{A}_2$,  $\mathrm{A}_3$ or $\mathrm{D}_4$ and let $C$ be the curve  in the pencil $|-K_{S}|$ that passes through 
$P$. 
\begin{enumerate}
\item[(3)]If the log pair $(S,D)$ is  not log canonical at  $P$, then
\end{enumerate}
\begin{itemize}
\item the log pair
$(S,C)$ is not log canonical at $P$;
\item the support of $D$ contains the support of $C$. 
\end{itemize}
\end{theorem}

\begin{proof}

Since $-K_S$ is ample, the first statement immediately follows from $-K_{S}\cdot D=1$

For a general member $Z$ in the anticanonical linear system $|-K_{S}|$, we have
$$
1=Z\cdot
D\geqslant\mathrm{mult}_{O}(Z)\mathrm{mult}_{O}(D)\geqslant
\mathrm{mult}_{O}(D).
$$
It then follows from  Lemma~\ref{lemma:mult} that the log pair $(S,D)$ is
log canonical at the base point $O$.

Now we consider a point $P$ on $S$ other than $O$. 
For (3) we first prove that 
$(S,D)$ is log canonical at $P$ if the support of $D$ does not contain the support of $C$. For this purpose, we suppose that
the support of the curve $C$ is not contained in the support of $D$. Note that $C$ is irreducible.

If $P$ is a smooth point, then we can obtain
$$
1=C\cdot
D\geqslant\mathrm{mult}_{P}(C)\mathrm{mult}_{P}(D)\geqslant
\mathrm{mult}_{P}(D),
$$
which implies that  $(S,D)$ is log canonical at $P$ by
Lemma~\ref{lemma:mult}.

Now we suppose that $P$ is a singular point of the surface $S$. Let $f\colon\tilde{S}\to S$ be the  minimal resolution of the singular point $P$.  Denote by $E_1,\ldots,E_r$ the
$f$-exceptional curves, denote by $\tilde{D}$ the proper
transform of the divisor $D$ on the surface $\tilde{S}$ and
denote by $\tilde{C}$ the proper transform of the curve $C$ on the
surface $\tilde{S}$. Then there are non-negative rational numbers
$a_1,\ldots,a_r$ such that
$$
K_{\tilde{S}}+\tilde{D}+\sum_{i=1}^{r}a_iE_i=f^*(K_{S}+D)\sim_{\mathbb{Q}}
0.
$$
We can immediately see how the proper transform  $\tilde{C}$ of the effective anticanonical divisor $C$  intersects the exceptional divisors $E_i$ (for instance, see \cite[Appendix]{Pa03}).

Suppose that  $P$  is a singular point of type $\mathrm{D}_4$. Then $r=4$ and we may assume that the exceptional divisor $E_3$ is the $(-2)$-curve that intersects all the other three $(-2)$-curves. We see from \cite[Appendix]{Pa03} that  $\tilde{C}\cdot E_3=1$ and $\tilde{C}\cdot
E_1=\tilde{C}\cdot E_2=\tilde{C}\cdot E_4=0$. We then obtain
$$
1-a_3=\left(f^{*}(-K_{S})-\sum_{i=1}^{r}a_iE_i\right)\cdot\tilde{C}=\tilde{D}\cdot\tilde{C}\geqslant 0.%
$$
Lemma~\ref{lemma:D4} therefore implies that   $(S,D)$ is log canonical at  $P$.

Suppose that $P$ is a singular point of type $\mathrm{A}_r$.
 We assume that $E_1$ and $E_r$ are the
tail curves, i.e., the $(-2)$-curves intersecting only one $(-2)$-curve, respectively. Then the curve $\tilde{C}$ intersects $E_1$ and $E_r$, respectively,  at one point  transversally (if $r=1$, then $\tilde{C}\cdot E_1=2$). But it does not intersect the other $(-2)$-curves. Therefore,
$$
1-a_1-a_r=\left(f^*(-K_{S})-\sum_{i=1}^{r}a_iE_i\right)\cdot \tilde{C}=\tilde{D}\cdot\tilde{C}\geqslant 0,%
$$
and hence $a_1+a_r\leqslant 1$ (if $r=1$, then
$a_1\leqslant\frac{1}{2}$).

Consider the case $r=1$.  Since $\tilde{D}\cdot E_1=2a_1\leqslant 1$, the log pair $(\tilde{S}, \tilde{D}+a_1E_1)$ is log canonical along the exceptional curve $E_1$ by Lemma~\ref{lemma:adjunction}. Therefore,  $(S, D)$ is log canonical at~$P$.

Next we consider the case $r=2$. We then have $a_1+a_2\leqslant 1$. Moreover, we obtain  $2a_1\geqslant a_2$ from the inequality
$$
2a_1-a_2=\tilde{D}\cdot E_1\geqslant 0.
$$
Similarly,
$2a_2\geqslant a_1$. Since $a_1+a_2\leqslant 1$, we may assume that $a_1\leqslant \frac{1}{2}$.
We obtain   $(\tilde{D}+a_2E_2)\cdot E_1=2a_1\leqslant 1$, and hence
$(\tilde{S}, \tilde{D}+ a_1E_1+a_2E_2)$ is log canonical along the curve $E_1$.
Furthermore, the inequality $$\tilde{D}\cdot E_2=2a_2-a_1\leqslant 2a_1+(a_2-a_1)=a_1+a_2\leqslant 1$$ implies that $(\tilde{S}, \tilde{D}+ a_1E_1+a_2E_2)$ is log canonical along the curve $E_2$. Consequently, $(S, D)$ is log canonical at~$P$.

Finally we consider the case $r=3$.  We have $a_1+a_3\leqslant 1$. Moreover, we may obtain $2a_1\geqslant a_2$, $2a_2\geqslant a_1+a_3$ and
$2a_3\geqslant a_2$ from
$$
\left\{\aligned%
&2a_1-a_2=\tilde{D}\cdot E_1\geqslant 0,\\
&2a_2-a_1-a_3=\tilde{D}\cdot E_2\geqslant 0, \\
&2a_3-a_2=\tilde{D}\cdot E_3\geqslant 0. \\
\endaligned
\right.
$$
We  may assume that $a_1\leqslant \frac{1}{2}$  since $a_1+a_3\leqslant 1$. Since $(\tilde{D}+a_2E_2+a_3E_3)\cdot E_1=2a_1\leqslant 1$,
the log pair $(\tilde{S}, \tilde{D}+ a_1E_1+a_2E_2+a_3E_3)$ is log canonical along the curve $E_1$.  
In addition,  $(\tilde{S}, \tilde{D}+ a_1E_1+a_2E_2+a_3E_3)$ is log canonical along $E_2\setminus E_3$  and $E_3\setminus E_2$  since $$
\left\{\aligned%
& \tilde{D}\cdot E_2=2a_2-a_1-a_3\leqslant 2(a_1+a_3)- (a_1+a_3)=a_1+a_3\leqslant 1, \\
& \tilde{D}\cdot E_3=2a_3-a_2\leqslant (2a_2-a_1)+a_3-a_2\leqslant a_1+a_3\leqslant 1.\\
\endaligned
\right.
$$
Let $Q$ be the intersection point of $E_2$ and $E_3$. 
We have
$$
\left\{\aligned%
& \tilde{D}\cdot E_2=2a_2-a_1-a_3\leqslant (4a_1-a_1+a_3)-2a_3=2a_1+(a_1+a_3)-2a_3\leqslant 2(1-a_3),\\
& \tilde{D}\cdot E_3=2a_3-a_2=2a_3+a_2-2a_2\leqslant 2a_3+2a_1-2a_2 \leqslant2(1-a_2). \\
\endaligned
\right.
$$
Since $\mathrm{mult}_{Q}(\tilde{D})\leqslant \tilde{D}\cdot E_3=2a_3-a_2\leqslant 1$,
Lemma~\ref{lemma:Trento} implies that  
$(\tilde{S},\tilde{D}+a_2E_2+a_3E_3)$  is log canonical at $Q$, and hence 
$(\tilde{S},\tilde{D}+a_1E_1+a_2E_2+a_3E_3)$  is log canonical at $Q$.
Consequently,   $(\tilde{S}, \tilde{D}+ a_1E_1+a_2E_2+a_3E_3)$ is log canonical along the three exceptional curves, and hence  $(S, D)$ is log canonical at  $P$.

If the log pair $(S, C)$ is log canonical at $P$, then we can obtain an effective anticanonical $\mathbb{Q}$-divisor $D_{\mu}$  from Lemma~\ref{lemma:convexity} such that $(S,D_{\mu})$  is not log canonical at $P$ and whose support does not contain the support of $C$. This however contradicts what we have proven so far. Therefore, $(S, C)$ is not log canonical at $P$.
\end{proof}

It is a common experience that  $\mathrm{D}_4$-singularity is more singular than $\mathrm{A}_4$-singularity. However, to our surprise, Theorem~\ref{theorem:dP1}~(3) does not hold for a  singular point of type $\mathrm{A}_4$  even though
every singular point of type $\mathrm{D}_4$ enjoys Theorem~\ref{theorem:dP1}~(3).

\subsection{Del Pezzo surfaces of degree $2$}
\label{sec:dP2}

Let $S$ be a del Pezzo surface of degree $2$ with at most ordinary
double points. Its anticanonical linear system $|-K_{S}|$ is base-point-free and induces a double cover $\pi\colon S\to \mathbb{P}^2$
ramified along a reduced quartic curve $R\subset\mathbb{P}^2$.
Moreover, the curve $R$ has at most ordinary double points. Note
that the curve $R$ may be reducible.

Let $D$ be an effective  anticanonical $\mathbb{Q}$-divisor on $S$.

\begin{lemma}
\label{lemma:dP2-isolated} Write
$$ D=\mu C+\Omega,$$
where $\mu$ is a non-negative rational number, $C$ is an irreducible curve and $\Omega$ is an effective $\mathbb{Q}$-divisor  whose support does not
contain the curve $C$.
If $\mu>1$, then $-K_{S}\cdot C=1$ and $\pi(C)$ is a
line in $\mathbb{P}^2$ that is an irreducible component of $R$.
\end{lemma}

\begin{proof}
Since $-K_S$ is ample, $-K_S\cdot C$ is a positive integer.
The equality $-K_{S}\cdot C=1$  immediately follows from
$$
2=-K_{S}\cdot\left(\mu C+\Omega\right)=-\mu K_{S}\cdot
C-K_{S}\cdot\Omega\geqslant-\mu K_{S}\cdot C>-K_{S}\cdot C.
$$
 This shows that $\pi(C)$ is
a line in $\mathbb{P}^2$.

Suppose that $\pi(C)$ is not an irreducible component of $R$. Then
there exists a curve $C^\prime$ different from $C$ such that $C+C^\prime\sim -K_{S}$ and
$\pi(C^\prime)=\pi(C)$. Write  $\Omega=\mu^\prime
C^\prime+\Delta$, where $\Delta$ is an effective
$\mathbb{Q}$-divisor on $S$ whose support does not contain the
curve $C^\prime$. Since $\mu>1$, we obtain $\mu'<1$. Therefore, by taking $\frac{1}{1-\mu'} \left(D-\mu'(C+C')\right)$ instead of $D$, we may assume that $\mu'=0$.

Since the intersection number $C\cdot C'$ belongs to $\frac{1}{2}\mathbb{Z}$, from $$1=D\cdot C'=\mu C\cdot C'+\Omega\cdot C' > C\cdot C',$$
we conclude that $C\cdot C'=\frac{1}{2}.$ Therefore, $C^2=-K_S\cdot C-C\cdot C'=\frac{1}{2}$. This implies that $C$ passes through three ordinary double points, which is impossible unless $\pi(C)$ is an irreducible component of $R$.
\end{proof}

\begin{theorem}\label{theorem:dP2-smooth-points}
Let $P$ be a smooth point of $S$.
\begin{enumerate}
\item\label{item:1} If $\pi(P)\not\in R$, then $(S, D)$ is log canonical at $P$.
\end{enumerate}
Suppose $\pi(P)\in R$ and let $T_P$ be the unique divisor in
$|-K_{S}|$ that is singular at  $P$.
\begin{enumerate}
\item[(2)]\label{item:2} If the log pair $(S,D)$ is  not log canonical at  $P$, then
\end{enumerate}
\begin{itemize}
\item the log pair
$(S,T_P)$ is not log canonical at $P$;
\item the support of $D$ contains the support of $T_P$. 
\end{itemize}
\end{theorem}

\begin{proof}
For (1), the proof of \cite[Lemma~3.2]{CheltsovParkWon} works verbatim even though we allow more than two ordinary double points.

For (2), we suppose $\pi(P)\in R$.
If the divisor $T_P$ is reduced, then  the proofs of \cite[Lemmas~3.4 and 3.5 ]{CheltsovParkWon}  verifies the statement.

If the divisor $T_P$ is not reduced,
 then $T_P=2C$ for some irreducible smooth curve $C$ and  $\pi(C)$ is a line in
$\mathbb{P}^2$ that is an irreducible component of the  quartic curve $R$.
It is clear that $(S, T_P)$ is not log canonical at $P$.
If $C\not\subset\mathrm{Supp}(D)$,  then
$$ 1=C\cdot
D\geqslant\mathrm{mult}_{P}(C)\mathrm{mult}_{P}(D)\geqslant\mathrm{mult}_{P}(D),%
$$
 and hence $(S, D)$ is log canonical at $P$ by Lemma~\ref{lemma:mult}. Therefore,  the support of $D$  must contain the support of $C$.  \end{proof}

Suppose that  $(S, D)$ is not log canonical at a  singular point $P$ of $S$.
Let $f\colon\bar{S}\to S$ be the blow up of  $S$ at $P$. Denote
by $E$ the $f$-exceptional curve and denote by $\bar{D}$ the
proper transform of the divisor $D$ on the surface $\bar{S}$.
Then
$$
\bar{D}=f^*(D)-aE\sim_{\mathbb{Q}}f^*(-K_{S})-aE
$$
for some positive rational number $a$. This gives
$$
K_{\bar{S}}+\bar{D}+aE=f^*(K_{S}+D)\sim_{\mathbb{Q}} 0,%
$$
which implies that  $(\bar{S}, \bar{D}+aE)$ is not
log canonical at some point $Q$ on  $E$ by
Remark~\ref{remark:log-pull-back}.

Let $H$ be a general curve in $|-K_{S}|$ that passes through $P$.
Denote by   $\bar{H}$ its proper transform on the surface
$\bar{S}$. Then $\bar{H}\cdot E=2$. We have
$$
0\leqslant\bar{H}\cdot\bar{D}=\bar{H}\cdot\left(f^*(-K_{S})-aE\right)=2-2a,
$$
which gives $a\leqslant 1$. Now applying
Lemma~\ref{lemma:adjunction} to $(\bar{S},aE+\bar{D})$ and
$E$, we get
$$
2a=E\cdot\bar{D}\geqslant\mathrm{mult}_{Q}\left(E\cdot\bar{D}\right)>1.
$$
Consequently, we see $\frac{1}{2}<a\leqslant 1$.
Since $a\leqslant 1$, the log pair $(\bar{S}, \bar{D}+aE)$ is
log canonical at every point of $E$ other than the point $Q$ by Remark~\ref{remark:log-pull-back}.

Since $-K_{\bar{S}}= f^*(-K_{S})$, the linear system
$|-K_{\bar{S}}-E|$ is a pencil. In fact, it is a base-point-free pencil.
A general curve in $|-K_{\bar{S}}-E|$ is a smooth rational curve
that intersects $E$ by two distinct points. Moreover, since
$|-K_{\bar{S}}-E|$ does not have any base points, there exists a
unique curve $C\in|-K_{S}|$ whose proper transform $\bar{C}$ by $f$ passes through the point $Q$.

\begin{theorem}\label{theorem:dP2} Let $P$ be an ordinary double point of $S$ and let $C$ be the curve  in $|-K_S|$ described above.
If the log pair $(S, D)$ is not log canonical at $P$, then \begin{itemize}
\item the log pair $(S, C)$ is not log canonical at $P$;
\item the support of $D$ contains the support of $C$.
\end{itemize}
\end{theorem}

\begin{proof} We suppose that  that $\mathrm{Supp}(D)$ does not contain
the support of $C$ and then look for a contradiction.
We have three cases  as below.

\medskip

\textbf{Case 1.} The curve $C$ is not reduced.

Then $C=2L$, where $L$ is a smooth
rational curve on $S$ such that $\pi(L)$ is a line in
$\mathbb{P}^2$  and it is an irreducible component of the curve $R$.

Denote by $\bar{L}$ the proper transform of the curve $L$ on the
surface $\bar{S}$. Then the point $Q$ belongs to $\bar{L}$ by the choice of $C$. Since
$L\not\subset\mathrm{Supp}(D)$, then
$$
1-a=\bar{L}\cdot\bar{D}\geqslant
\mathrm{mult}_{Q}\left(\bar{D}\right),
$$
and hence that $1\geqslant a+\mathrm{mult}_{Q}(\bar{D})=\mathrm{mult}_{Q}(\bar{D} +a E)$.
Therefore,  $(\bar{S}, \bar{D}+aE)$ is log
canonical at~$Q$ by Lemma~\ref{lemma:mult}. This is a contradiction. 
\medskip

\textbf{Case 2.} The curve $C$ is reduced and irreducible.

 Put
$m=\mathrm{mult}_{Q}(\bar{D})$. From
$$
2-2a=\bar{C}\cdot\bar{D}\geqslant
m,%
$$
we obtain $m+2a\leqslant 2$. Note that $m\leqslant 2-2a<1$ since $a>\frac{1}{2}$.

Let $g\colon\check{S}\to\bar{S}$ be the blow up of the surface $\bar{S}$ at the point $Q$.
Denote by $F$ the $g$-exceptional curve  and
denote by $\check{E}$ and $\check{D}$ the proper transforms of the divisors $E$ and $\bar{D}$ on the
surface $\check{S}$, respectively. Then
$$
K_{\check{S}}+\check{D}+a\check{E}+(a+m-1)F=g^*\left(K_{S}+aE+D\right),
$$
and 
$(\check{S},\check{D}+a\check{E}+(a+m-1)F)$ is not log canonical at some
point $O$ of the exceptional curve~$F$.

Since $a+m-1\leqslant 1$, the inequality
$$
\mathrm{mult}_{O}\left(\check{D}\right)\leqslant F\cdot \check{D}=m\leqslant 1
$$
implies that $(\check{S},\check {D}+(a+m-1)F)$ is  log canonical along the divisor $F$ by Lemma~\ref{lemma:adjunction}. Therefore, the point $O$ must be the intersection point of $F$ and $\check{E}$.

Since
$\mathrm{mult}_{O}(\check{D})\leqslant\mathrm{mult}_{Q}(\bar{D})=m\leqslant
1$, we can apply
Lemma~\ref{lemma:Trento} to the log pair
$(\check{S},\check{D}+a\check{E}+(a+m-1)F)$  at the point $O$, so that we obtain either
$$
2a-m=\check{D}\cdot\check{E}>2(2-a-m) \ \ \ \mbox{or} \ \ \ m=\check{D}\cdot F>2(1-a).$$
However, both the inequalities are impossible
since $m+2a\leqslant 2$. This is a contradiction.

\medskip

\textbf{Case 3.} The curve $C$ is reduced but reducible.

The curve $C$ consists of two distinct smooth irreducible and reduced curves $L_1$ and $L_2$.  Note that  $-K_{S}\cdot L_1=-K_{S}\cdot L_2=1$ and these two curves intersect at the point $P$. Without loss of
generality, we may assume that the curve $L_1$ is not contained in the support of $D$.
Then we put $D=bL_2+\Omega$, where $b$ is a non-negative rational
number and $\Omega$ is an effective $\mathbb{Q}$-divisor on $S$ whose support does not contain the curve $L_2$.
Denote by $\bar{L}_1$, $\bar{L}_2$ and $\bar{\Omega}$ the proper transforms of
the curves $L_1$, $L_2$ and the divisor $\Omega$ on the surface $\bar{S}$,
respectively.
Note that  $\bar{L}_1\cdot E=\bar{L}_2\cdot E=1$.

The point $Q$ cannot belong to the curve $\bar{L}_1$. Indeed, if so, then
$$
1-a=\bar{L}_1\cdot\bar{D}\geqslant
\mathrm{mult}_{Q}\left(\bar{D}\right),
$$
and hence $\mathrm{mult}_{Q}\left(\bar{D} +aE\right)\leqslant 1$. Since
 $(\bar{S},
\bar{D}+aE)$ is  not log canonical at $Q$, this is absurd. Therefore,
the point $Q$ must belong to the curve $\bar{L}_2$.

Recall that $\pi(L_1)=\pi(L_2)$ is a line in $\mathbb{P}^2$ that
passes though the point $\pi(P)$. Since $Q\not\in\bar{L}_1$ and  $Q\in\bar{L}_2$, the
intersection $L_1\cap L_2$ consists of two distinct points, one of which is the point $P$. Thus, the intersection
$\bar{L}_1\cap\bar{L}_2$ consists of a single point. This
point can be either  a smooth point or an ordinary double point of the surface
$\bar{S}$. In the former case, we have
$\bar{L}_1\cdot\bar{L}_2=1$ and
$\bar{L}_1^2=\bar{L}_2^2=-1$. In the latter case, we have
$\bar{L}_1\cdot\bar{L}_2=\frac{1}{2}$ and
$\bar{L}_1^2=\bar{L}_2^2=-\frac{1}{2}$.

From
$$
1-a-b(\bar{L}_1\cdot\bar{L}_2)=\bar{\Omega}\cdot\bar{L}_1\geqslant 0%
$$
we obtain $a+b(\bar{L}_1\cdot\bar{L}_2)\leqslant 1$.
Therefore,
$$\bar{L}_{2}\cdot\bar{\Omega}=1-a-b\bar{L}_2^2\leqslant (1-a)\left(1-\frac{\bar{L}_2^2}{\bar{L}_1\cdot \bar{L}_2}\right)=2(1-a)$$
and
$$E\cdot\bar{\Omega}=2a-b \leqslant 2- b(2\bar{L}_1\cdot\bar{L}_2+1)\leqslant 2(1-b).$$

Meanwhile, from $$
\left\{\aligned%
&1-a-b\bar{L}_2^2=\bar{L}_{2}\cdot\bar{\Omega}\geqslant \mathrm{mult}_{Q}(\bar{\Omega}),\\
&2a-b=E\cdot\bar{\Omega}\geqslant \mathrm{mult}_{Q}(\bar{\Omega}),\\
\endaligned
\right.
$$
we obtain $ 2\mathrm{mult}_{Q}(\bar{\Omega}) \leqslant 1+a-(1+\bar{L}_2^2)b\leqslant 1+a$.
Therefore, $\mathrm{mult}_{Q}(\bar{\Omega}) \leqslant 1$. This enables us to apply
Lemma~\ref{lemma:Trento} to 
 $(\bar{S}, \bar{\Omega}+aE+b\bar{L}_2)$ at the point $Q$.  The log pair $(\bar{S}, \bar{\Omega}+aE+b\bar{L}_2)$ must be log canonical at  $Q$. This is a contradiction.
 \medskip
 
 These three cases lead to the conclusion that the support of $D$ contain the curve $C$ if $(S, D)$ is not log canonical at $P$.
\medskip

In Case~1, it is obvious that $(S, C)$ is  not log canonical at $P$.
 In Cases~2 and 3, if $(S, C)$ is log canonical at $P$, then we can obtain an effective anticanonical $\mathbb{Q}$-divisor $D_{\mu}$  from Lemma~\ref{lemma:convexity} such that $(S,D_{\mu})$  is not log canonical at $P$ and whose support does not contain the support of $C$. This however contradicts what we have proven in Cases 2 and 3. Therefore, $(S, C)$ is not log canonical at $P$.
\end{proof}

\subsection{Proof of Theorem~\ref{theorem:main} (I)}
\label{sec:cylinders}
Now we are ready to prove Theorem~\ref{theorem:main} (I), i.e.,  if a surface $S$ is either a del Pezzo surface of degree $2$ with only ordinary double points or
a del Pezzo surface of degree $1$ with du Val singularities of types $\mathrm{A}_1$, $\mathrm{A}_2$, $\mathrm{A}_3$,
$\mathrm{D}_4$ only, then it
cannot admit any
$(-K_{S})$-polar cylinder.

To this end, we suppose that the del Pezzo surface $S$ contains a
$(-K_{S})$-polar cylinder and then we look for a
contradiction.

Since $S$ contains a $(-K_{S})$-polar cylinder, there is
 an effective anticanonical
$\mathbb{Q}$-divisor $D$
such that $U=S\setminus\mathrm{Supp}(D)$ is isomorphic to an affine variety $ Z\times
\mathbb{A}^1$ for some smooth rational affine curve $Z$. Put
$D=\sum_{i=1}^{r}a_i D_i$, where each $D_i$ is an irreducible and
reduced curve and each $a_i$ is a positive rational number.

\begin{lemma}
\label{lemma:rank-Cl-group}  The components of $D$ generate  the divisor class group
$\mathrm{Cl}(S)$ of the surface $S$. In particular,  the number of the irreducible components of $D$ is at least the rank of $\mathrm{Cl}(S)$.
\end{lemma}
\begin{proof}
See
\cite[Lemma~4.6]{KPZ12b}.
\end{proof}

The natural projection $U\cong Z\times\mathbb{A}^1\to Z$ induces a
rational map $\phi\colon S\dasharrow\mathbb{P}^1$.
Denote by $\mathcal{L}$ the
 pencil  on the surface $S$ that induces the rational map $\phi$.
Then either the pencil
$\mathcal{L}$ is base-point-free or its base locus consists of a
single point.

\begin{lemma}
\label{lemma:base-point} The pencil $\mathcal{L}$ is not
base-point-free.
\end{lemma}

\begin{proof}
Suppose that the pencil $\mathcal{L}$ is base-point-free. Then
$\phi$ is a morphism, which implies that there exists exactly one
irreducible component of $\mathrm{Supp}(D)$ that does not lie in
the fibers of $\phi$. Moreover, this irreducible component is a section. Without loss of
generality, we may assume that this component is $D_r$. Let $L$ be
a sufficiently general curve in  $\mathcal{L}$. Then
$$
2=-K_{S}\cdot L=D\cdot L=\sum_{i=1}^{r}a_i D_i\cdot L=a_rD_r\cdot
L,
$$
and hence $a_r=2$.

By Theorem~\ref{theorem:dP1}~(1), the surface $S$ cannot be of degree $1$, and hence it must be of degree $2$.
Then the anticanonical linear system $|-K_{S}|$ is base-point-free and
induces a double cover $\pi\colon S\to \mathbb{P}^2$ ramified
along a reduced quartic curve $R\subset\mathbb{P}^2$. Moreover,
the curve $R$ has at most ordinary double points. Furthermore, it
follows from Lemma~\ref{lemma:dP2-isolated} that $\pi(D_r)$ is a line in $\mathbb{P}^2$ that is an
irreducible component of $R$. Therefore,  $-K_{S}\sim 2D_r$, and hence $r=1$.
This  implies that the rank of the divisor class group of $S$ is
one by Lemma~\ref{lemma:rank-Cl-group}. However,
since the curve $R$ has at most six
singular points, the surface $S$ can attain at most six ordinary double points. Therefore, the rank of the divisor class group of $S$ is at least two. This is a contradiction.
\end{proof}

Denote the unique base point of the pencil $\mathcal{L}$  by $P$.
Resolving the base locus of the pencil $\mathcal{L}$ we obtain a
commutative diagram
$$
\xymatrix{ &W\ar[dl]_{f_1}\ar[dr]^{f_2}& \\
S\ar@{-->}[rr]^{\phi}&&\mathbb{P}^1,}
$$
where $f_1$ is a composition of blow ups at smooth points over the
point $P$ and $f_2$ is a morphism whose general fiber is a smooth
rational curve. Denote by $E_1, \ldots, E_n$ the exceptional
curves of the birational morphism $f_1$. Then there exists exactly
one curve among them that does not lie in the fibers of the
morphism $f_2$. Without loss of generality,
we may assume that this curve is $E_n$. The curve $E_n$ is a section of
the morphism $f_2$.

For every $D_i$, denote by $\hat{D}_i$ its proper transform on the
surface $W$. Every curve $\hat{D}_i$ lies in a fiber of the
morphism $f_2$.

\begin{lemma}
\label{lemma:pseudo-convexity} Let $D^\prime$ be an effective anticanonical 
$\mathbb{Q}$-divisor on $S$ with $\mathrm{Supp}(D^\prime)\subseteq\mathrm{Supp}(D)$. Denote
by $\hat{D}^\prime$ its proper transform on $W$. Then
$$
K_{W}+\hat{D}^\prime=f_1^{*}\left(K_{S}+D^\prime\right)+\sum_{i=1}^{n}c_iE_i\sim_{\mathbb{Q}}\sum_{i=1}^{n}c_iE_i%
$$
 for some rational numbers $c_1, \ldots, c_n$.
Moreover, we have $c_n=-2$. In particular, the log pair
$(S,D^\prime)$ is not log canonical at the point $P$.
\end{lemma}

\begin{proof}
The existence of rational numbers $c_1, \ldots, c_n$ is
obvious. We must show that $c_n=-2$. Write
$D^\prime=\sum_{i=1}^{r}b_i D_i$, where each $b_i$
is a non-negative rational number. Let $L$ be a
sufficiently general fiber of the morphism $f_2$. Then
$$
-2=K_{W}\cdot L=K_{W}\cdot L+\sum_{i=1}^{r}b_i \hat{D}_i\cdot L=\sum_{i=1}^{n}c_iE_i\cdot L=c_n,%
$$
because $E_n$ is a section of the morphism $f_2$, every curve
$\hat{D}_i$ lies in a  fiber of the morphism $f_2$ and every curve
$E_i$ with $i<n$ also lies in a fiber of the morphism $f_2$. Hence,
$c_n=-2$ and~$(S,D^\prime)$ is not log canonical
at $P$.
\end{proof}

Lemma~\ref{lemma:pseudo-convexity} shows that  if a del Pezzo surface $S'$ with at worst du Val singularities  contains a
$(-K_{S'})$-polar cylinder, then the surface $S'$ must possess
an effective anticanonical $\mathbb{Q}$-divisor $B$ such that the log pair $(S', B)$ is not log canonical. Such an effective anticanonical
$\mathbb{Q}$-divisor is called a  \emph{tiger} on the del Pezzo surface $S'$ (cf. \cite{KeMc99}).
In particular, applying Lemma~\ref{lemma:pseudo-convexity} to~$(S,D)$, we see
that~$(S,D)$ is not log canonical at $P$.
\bigskip

\emph{Proof of Theorem~\ref{theorem:main} (I).}
\medskip

\textbf{Case 1.}  The surface $S$ is of degree $1$.

By Theorem~\ref{theorem:dP1}~(2), $P$ is not the
base point of the pencil $|-K_{S}|$. Thus, there exists a unique
curve $C$ in the pencil $|-K_{S}|$ that passes through  $P$.
If the rank of the divisor class group of $S$ is greater  than one,  then
$D\ne C$ by Lemma~\ref{lemma:rank-Cl-group}. If the rank is one, then the open set $S\setminus C$ must contain a singular point.
But $S\setminus\mathrm{Supp}(D)\cong\mathbb{A}^1\times Z$ is smooth. Therefore, $D\ne C$.

Let $\mu$ be the greatest rational number such that
$D^\prime=(1+\mu)D-\mu C$ is effective. Then $(S,D^\prime)$ is not log canonical at  $P$ by
Lemma~\ref{lemma:pseudo-convexity}. This contradicts
Theorem~\ref{theorem:dP1}~(3).

\medskip

\textbf{Case 2.}  The surface $S$ is of degree $2$.

We have the double cover $\pi\colon S\to
\mathbb{P}^2$ ramified along a reduced quartic curve
$R\subset\mathbb{P}^2$ given by the anticanonical linear system.

The surface $S$ has at most six ordinary double points. If it has six ordinary double points, then the quartic curve $R$ consists  of
four distinct lines on $\mathbb{P}^2$. In other words,
the rank of the divisor class group of $S$ is at least two  and if it is two, then the quartic curve $R$ consists  of
four distinct lines on $\mathbb{P}^2$.

Note that the point $\pi(P)$ must belong to the quartic curve $R$
 by
Theorem~\ref{theorem:dP2-smooth-points}~(\ref{item:1}).

Suppose that the quartic curve $R$ is smooth at $\pi(P)$.
Let $T_P$ be the unique curve in $|-K_{S}|$ that is singular at
$P$. Then the log pair $(S,T_P)$ is not log canonical at $P$ by
Theorem~\ref{theorem:dP2-smooth-points}~(2).

The curve $T_P$
 consists of
at most two irreducible components. Thus, if
the rank of the divisor class group of $S$ is at least $3$, then $D\ne T_P$ by
Lemma~\ref{lemma:rank-Cl-group}. If
the rank of the divisor class group of $S$ is two, then $R$ is a union of four
distinct lines. This implies that the support of $T_P$ is
 an
irreducible curve. Therefore, $D\ne T_P$ by
Lemma~\ref{lemma:rank-Cl-group}.

Let $\mu$ be the greatest rational number such that
$D^\prime=(1+\mu)D-\mu T_P$ is effective. Then~$(S,D')$
is not log canonical at $P$ by
Lemma~\ref{lemma:pseudo-convexity}. This contradicts
Theorem~\ref{theorem:dP2-smooth-points}~(2).

Thus, the point $P$ must be  a singular point of the surface $S$. Let
$f\colon\bar{S}\to S$ be the blow up of the surface $S$ at $P$. Then there
exists a commutative diagram
$$
\xymatrix{ &&&W\ar[ddl]_{f_1}\ar[dlll]_t\ar[ddr]^{f_2}& \\
\bar{S}\ar[drr]_f&&&&\\
&&S\ar@{-->}[rr]^{\phi}&&\mathbb{P}^1,}
$$
where $t$ is a birational morphism. Denote by $E$ the
$f$-exceptional curve and denote by $\bar{D}$ the proper
transform of the divisor $D$ on the surface $\bar{S}$. The image of $E_n$ by the birational morphism $t$ is a point on the exceptional curve $E$. Denote this point by $Q$.

Note that the log pair $(\bar{S}, f^*(D))$ is not log canonical at  $Q$.

Let $C$ be the unique curve in the anticanonical linear system $|-K_S|$ whose proper transform  by the blow up $f$ passes through the point $Q$.

The curve $C$ has at most two irreducible components. If the curve $C$ is irreducible, then $D\ne C$
by Lemma~\ref{lemma:rank-Cl-group}. Suppose that the curve $C$ has two irreducible components.
If the rank of the divisor class group of $S$ is greater than two,  then
$D\ne C$ by Lemma~\ref{lemma:rank-Cl-group}.
If the rank of the divisor class group of $S$ is two, then $R$ is a union of four
distinct lines, and hence  $(S,C)$ is log canonical. By
Theorem~\ref{theorem:dP2},  $(S,D)$ must be log
canonical as well. This is a contradiction. 

Let $\mu$ be the greatest rational number such that
$D^\prime=(1+\mu)D-\mu C$ is effective. The log
pair $(S,D^\prime)$ is not log canonical at $P$ and   $(\bar{S},f^*(D^\prime))$ is 
not log canonical at $Q$ by
Lemma~\ref{lemma:pseudo-convexity}. The same curve $C$ is the curve in $|-K_S|$ 
whose proper transform  by the blow up $f$ passes through $Q$.
By Theorem~\ref{theorem:dP2} the curve $C$ is contained in the support of $D'$. This is a contradiction. 
\hfill$\square$

\begin{remark}\label{remark:tiger} In order to prove Theorem~\ref{theorem:main}~(I), we use Theorems~\ref{theorem:dP1},~\ref{theorem:dP2-smooth-points} and \ref{theorem:dP2}. In fact, they show that under the condition of Theorem~\ref{theorem:main}~(I) the support of each tiger
on $S_d$ contains 
the support of an effective anticanonical  divisor. Furthermore, on a del  Pezzo surface $S_{d}$ not listed in 
Theorem~\ref{theorem:main}~(I) we can always find a tiger whose support does not contain the support of any effective anticanonical divisor. Indeed,  using the effective anticanonical divisors listed in \cite{PaW11} and the cylinders constructed in the present article, we have hunted such tigers. Since this case-by-case hunting is a tedious job, we do not present such tigers here.   We may rephrase Theorem~\ref{theorem:main} as follows:\\
\emph{Let $S$ be a del Pezzo surface with at most du Val singularities. It has no $(-K_S)$-cylinder if and only if the support of each tiger  on $S$ contains 
the support of an effective anticanonical  divisor.}
\end{remark}

\section{Construction of cylinders}
In this section, we prove  Theorem~\ref{theorem:main}~(II). For a given singular del Pezzo surface $S$ not listed in Theorem~\ref{theorem:main} (I) we find an effective anticanonical $\mathbb{Q}$-divisor $D_S$ such that the complement of the support of $D_S$ is isomorphic to $\mathbb{A}^1\times Z$ for some smooth rational affine curve $Z$.

\subsection{Construction for high degrees}

We first construct $(-K_S)$-polar cylinders for singular del Pezzo surfaces of degree $3$ with only du Val singularities. Using this construction, we also obtain $(-K_S)$-polar cylinders for singular del Pezzo surfaces of degrees  $\geq 4$.

\begin{theorem}\label{theorem:cubic-projection}
Let $S$ be a singular cubic surface in $\mathbb{P}^3$ with  only du Val singularities. Let $P$ be a singular point of $S$ and let $L_1,\ldots, L_r$  be the lines on $S$ passing through $P$.
\begin{itemize}
\item There are  positive rational numbers $a_1,\ldots, a_r$  such that  the effective $\mathbb{Q}$-divisor $$a_1L_1+\ldots+a_rL_r$$ is 
$\mathbb{Q}$-linearly equivalent to $-K_S$.
\item If $P$ is an ordinary double point, for a hyperplane section $L$ of $S$ such that it has a cuspidal point at $P$ the set $$S\setminus (L\cup L_1\cup\ldots\cup L_r)$$ is a cylinder.
\item If $P$ is a non-ordinary double point, then the set $$S\setminus (L_1\cup\ldots\cup L_r)$$ is a cylinder.
\end{itemize}
In particular, $S$ has a $(-K_S)$-polar cylinder.
\end{theorem}
\begin{proof}
It is easy to see that there are lines $L_1,\ldots, L_r$ passing through $P$ and $1\leq r\leq 6$.

Let $\pi: S\dasharrow \mathbb{P}^2$ be the projection from $P$. It is a birational map and it contracts exactly  the lines $L_1,\ldots, L_r$. Let $f:\bar{S}\to S$ be the blow up at the point $P$ and $E$ be its exceptional divisor. Then the map $g:=\pi \circ f:\bar{S}\to \mathbb{P}^2$ is defined everywhere. Let $C$ be the image of $E$ by $g$ on $\mathbb{P}^2$. It  is a conic curve. Furthermore, it contains all the points $\pi(L_i)$.
If $P$ is an ordinary double point, then $C$ is a smooth conic. If $P$ is of type $\mathrm{A}_n$, $n\geq 2$, then $C$ consists of two distinct lines. If $P$ is of type either $\mathrm{D}_n$ or $\mathrm{E}_n$, then $C$ is a (double) line.

Since the map $\pi$ is defined by the linear system of hyperplane sections passing through $P$, the pull-back of $\frac{1}{2}C$ by $\pi$ is an effective anticanonical $\mathbb{Q}$-divisor.
Moreover,  we immediately see
\[\pi^*(\frac{1}{2}C)=a_1L_1+\ldots+a_rL_r\]
for some positive rational numbers $a_i$  because the curve $C$ passes through all the points $\pi(L_1),\ldots, \pi(L_r)$.

It is easy to see 
\[S\setminus (L_1\cup\ldots\cup L_r)\cong 
\bar{S} \setminus (E\cup \bar{L}_1\cup\ldots\cup \bar{L}_r) \cong
\mathbb{P}^2\setminus \mathrm{Supp}(C),\]
where $\bar{L}_i$ is the proper transform of $L_i$ by $f$.

Suppose that $P$ is a non-ordinary double point.  Since the support of  $C$ consists of at most two lines, its complement is a cylinder.

Suppose that $P$ is an ordinary double point. Then the conic $C$ is smooth. 
Let $L$ be a hyperplane section of $S$ that has a cuspidal point at $P$. 
For a hyperplane section of $S$ to have a cuspidal point, it has to be irreducible. Therefore, the hyperplane section $L$ cannot meet $L_i$ at a point other than $P$. 

The proper transform of $L$ by $f$ is contained in the smooth locus of the surface $\bar{S}$ and it is  a smooth curve that meets the exceptional curve $E$ tangentially at a single point. Note that it does not meet any of $\bar{L}_i$'s. Therefore, its image by $g$ is a line tangent to the curve $C$. Note that the pull-back of a line tangent to $C$ at a point other than $\pi(L_i)$  by the map $\pi$ is such a hyperplane section as $L$. Consequently, the set 
 $$S\setminus (L\cup L_1\cup\ldots\cup L_r)$$ is a cylinder. This completes the proof.
\end{proof}

\begin{theorem}
Let $S$ be a del Pezzo  surface of degree $d\geq 4$ with  at worst du Val singularities. Then it always contains
a $(-K_S)$-polar cylinder. 
\end{theorem}
\begin{proof}
Recall that we consider only  del Pezzo surfaces of degrees $\leq 7$ 
for the reason explained right after Corollary~\ref{corollary:main}.

Let $\tilde{S}$ be the minimal resolution of $S$ and let  $A$ be a $(-1)$-curve on $\tilde{S}$. 
Let $\rho_1:\Sigma_1\to \tilde{S}$ be the blow up of $\tilde{S}$ at a general point on $A$ and let $E_1$ be its exceptional curve. Let $\rho_2:\Sigma_2\to \Sigma_1$ be the blow up of $\Sigma_1$ at a general point on $E_1$ and let $E_2$ be its exceptional curve. We repeat this procedure $(d-3)$ times until we get a weak del Pezzo surface $\Sigma_{d-3}$ of degree $3$.  Set $\rho=\rho_{d-3}\circ\cdots\circ \rho_1$ and denote the proper transforms of $A$ and $E_1,\ldots, E_{d-4}$ by the same symbols. Contract all the $(-2)$-curves on 
$\Sigma_{d-3}$ to get a del Pezzo surface $\Sigma$ of degree $3$ with du Val singularities. This contraction is denoted by 
$\psi:\Sigma_{d-3}\to\Sigma$. 

The image of $A$ by $\psi$ is a singular point of $\Sigma$. We apply Theorem~\ref{theorem:cubic-projection} to this singular point to get an effective anticanonical $\mathbb{Q}$-divisor $D$ on $\Sigma$ that defines a cylinder. This divisor contains all the lines passing through the point $\psi(A)$. Then
\[\psi^*(D)\sim_\mathbb{Q}-K_{\Sigma_{d-3}}.\]
We may write 
\[\psi^*(D)=aA+\sum^{d-3}_{i=1}b_iE_i+\Delta,\]
where $a$ and $b_i$'s are positive rational numbers and $\Delta$ is an effective $\mathbb{Q}$-divisor on $\Sigma_{d-3}$ whose support  contains none of $A$ and $E_i$'s. Note that $\psi(E_{d-3})$ is a line passing through the point~$\psi(A)$.

The divisor $aA+\rho (\Delta)$ on $S$ is $\mathbb{Q}$-linearly equivalent to $-K_{S}$. Furthermore,
\[S\setminus \mathrm{Supp}\left(aA+\rho (\Delta)\right)\cong 
\Sigma_{d-3} \setminus \mathrm{Supp}\left(aA+\sum^{d-3}_{i=1}b_iE_i+\Delta\right) \cong
S\setminus \mathrm{Supp}\left(D\right).\]
Therefore, $aA+\rho (\Delta)$ defines a $(-K_S)$-polar cylinder on $S$.
\end{proof}

\subsection{Construction for low degrees}\label{subsection:low degrees}
Now we seek for an effective anticanonical $\mathbb{Q}$-divisor~$D_S$ that defines a  cylinder  on a given singular del Pezzo surface $S$ of degree $\leq 2$ not listed in Theorem~\ref{theorem:main} (I).
To this end, instead of the singular surface $S$, we can consider its minimal resolution $f:\tilde{S}\to S$. Since we only allow du Val singularities on the surface $S$, the surface $\tilde{S}$ is a smooth weak del Pezzo surface, i.e., a smooth surface with nef and big anticanonical class $-K_{\tilde{S}}$.
On this smooth  weak del Pezzo surface, it is enough to find an effective anticanonical $\mathbb{Q}$-divisor $D_{\tilde{S}}$ satisfying the following conditions:
\begin{itemize}
\item its support contains all the $(-2)$-curves on $\tilde{S}$;
\item the complement of the support of $D_{\tilde{S}}$ is isomorphic to $\mathbb{A}^1\times Z$ for some smooth rational affine curve $Z$.
\end{itemize}
Then we can take the divisor $D_S$ as $f(D_{\tilde{S}})$.

On the other hand, in order to find such a divisor $D_{\tilde{S}}$, we start with the projective plane $\mathbb{P}^2$ and one of the following   effective anticanonical  $\mathbb{Q}$-divisors $D_{\mathbb{P}^2}$  on $\mathbb{P}^2$:
\begin{itemize}
\item a triple line $3L$;
\item  $a_1L_1+a_2L_2$, where $a_1+a_2=3$ and $L_1$, $L_2$ are distinct lines;
\item $aL+bC$, where $a+2b=3$, $C$ is an irreducible conic and $L$ is a line tangent to the conic $C$;
\item $a_1L_1+a_2L_2+a_3L_3$, where $a_1+a_2+a_3=3$ and $L_1$, $L_2$, $L_3$ are  three distinct lines meeting at a single point.
\end{itemize}
Note that the complement $\mathbb{P}^2\setminus \mathrm{Supp}(D_{\mathbb{P}^2})$ is isomorphic to $\mathbb{A}^2$,
$\mathbb{A}^1\times\left(\mathbb{A}^1\setminus \{\mbox{one point}\}\right)$, $\mathbb{A}^1\times\left(\mathbb{A}^1\setminus \{\mbox{one point}\}\right)$, $\mathbb{A}^1\times\left(\mathbb{A}^1\setminus \{\mbox{two points}\}\right)$, respectively.

Let $S$ be a given del Pezzo surface with du Val singularities and $\tilde{S}$ be its minimal resolution.
Starting from $\mathbb{P}^2$ with one of the divisors $D_{\mathbb{P}^2}$ we build  a sequence of blow ups $h:\check{S}\to \mathbb{P}^2$ and a sequence of blow downs  $g:\check{S}\to \tilde{S}$ with the following properties.
Let $D_{\check{S}}$ be the log pull-back of $D_{\mathbb{P}^2}$ by $h$, i.e., the divisor such that 
$$K_{\check{S}}+D_{\check{S}}=h^*(K_{\mathbb{P}^2}+D_{\mathbb{P}^2}).$$
The divisor $D_{\check{S}}$
satisfies the following:
\begin{enumerate}
\item it is effective;
\item its support contains all the exceptional curves of $h$;
\item its support contains all the curves contracted by $g$.
\end{enumerate}

$$
\xymatrix{ &(\check{S}, D_{\check{S}})\ar[ddl]_h\ar[dr]^g& \\
&&(\tilde{S}, D_{\tilde{S}})\ar[d]^f\\
(\mathbb{P}^2, D_{\mathbb{P}^2})&& (S, D_S)}
$$

Existence of such birational morphisms $h$ and $g$ shows that the given surface $S$ admits a $(-K_S)$-polar cylinder since
$$ \mathbb{P}^2\setminus \mathrm{Supp}(D_{\mathbb{P}^2})\cong \check{S}\setminus \mathrm{Supp}(D_{\check{S}})
\cong \tilde{S}\setminus \mathrm{Supp}(D_{\tilde{S}})\cong S\setminus \mathrm{Supp}(D_{S}),$$
where $D_{\tilde{S}}=g(D_{\check{S}})$ and $D_{S}=f(D_{\tilde{S}})$.

For a given del Pezzo surface of degree $\leq 2$ with du Val singularities not listed in Theorem~\ref{theorem:main}~(I), the method to construct such birational morphisms $h$ and $g$
is described in Tables~\ref{table:deg1} and~\ref{table:deg2} at the end.

\subsection{The Table}

 For a given del Pezzo surface $S$ of degree $\leq 2$ with du Val singularities,  in Tables~\ref{table:deg1}~and~\ref{table:deg2}, we provide the divisor $D_{\mathbb{P}^2}$ and the birational morphisms $h$ and $g$ described in \ref{subsection:low degrees} in order to
show how to construct a $(-K_S)$-polar cylinder on $S$.

 We read the tables in the following way.
 
 In the first column the singularity types are given in normal size letters. The singularity types in small letters 
 in Table~\ref{table:deg1}  are those for del Pezzo surfaces of degree $2$. These singularity types in small letters will be explained later.

 The birational morphism $h$ is obtained by successive blow ups with exceptional curves $E_{\one},\ldots,  E_{\thirteen}$ in this order. The configuration of these exceptional curves given in the third column shows how to take these blow ups. The exceptional curves  $E_{\one},\ldots, E_{\thirteen}$ are labelled by \ding{172}, ... ,  {\stiny \blowup{\textbf{13}}}, respectively, in the third column.
 The configuration in the third column also shows $D_{\mathbb{P}^2}$.
We denote the proper transforms of lines from $\mathbb{P}^2$ by $L_i$ (or $L$). We denote the proper transform of an irreducible conic  from $\mathbb{P}^2$ by $Q$.

In the second column, the sum of  the first divisor (tiger) and the second divisor (divisor contracted), if any, is the divisor $D_{\check{S}}$.
If we have the second divisor  in the second column,
the birational morphism $g$ is obtained by contracting curves drawn by dotted curves in the third column. 
The second divisor in the second column is contracted by $g$. Indeed, each component of the second divisor is depicted by a dotted curve in the third column.
If we do not have the second divisor  in the second column, then $\check{S}=\tilde{S}$ and the morphism $g$ is the identity.
The fat curves  in  the third column  are the curves to be $(-2)$-curves on $\tilde{S}$. The thin lines with dots at one of the ends are the curves to be $(-1)$-curves on $\tilde{S}$.  The  wiggly lines are  the curves to be non-negative curves on $\tilde{S}$.

In  the second column, the curves without superscripts  are $(-2)$-curves on $\check{S}$. The curves superscripted by black-circled numbers are the smooth rational curves  on $\check{S}$ with self-intersection numbers of the negatives of the black-circled numbers.
The curves superscripted by the circled numbers are the smooth rational curves  on $\check{S}$ with self-intersection numbers of the circled numbers.

For a del Pezzo surface of degree $2$ with a singularity type written in small letters in Table~\ref{table:deg1}
the divisor $D_{\mathbb{P}^2}$ and the birational morphisms $h$ and $g$ can be easily  obtained by contracting 
one of the $(-1)$-curves (thin lines with dots at one of the ends) in the third column. 
 Only for singularity types $\textnormal{D}_4$, $\textnormal{A}_3$ and $\textnormal{A}_2$ they  cannot be obtained in this way. For these three types, we provide the divisor $D_{\mathbb{P}^2}$ and the birational morphisms $h$ and $g$  in Table 2, separately.

\bigskip
The methods  are given according to the singularity types of  singular del Pezzo surfaces. Even though they  show how to construct the birational morphisms $h$ and $g$  for \emph{a seemingly single} del Pezzo surface $S$ of a given singularity type,  they indeed demonstrate how to obtain the birational morphisms $h$ and $g$ for \emph{every} del Pezzo surface $S$ of a given singularity type.
Let us explain the reason.

Let $S'$ be an arbitrary del Pezzo surface  of a given singularity type and $\tilde{S}'$ be its  minimal resolution. 
The configurations of $(-1)$-curves and $(-2)$-curves on smooth weak del Pezzo surfaces  are the same if the surfaces are of the same type.
If the divisor $D_{\tilde{S}}$ in the table for the given singularity type consists of only negative curves, then we can immediately find  a $\mathbb{Q}$-divisor $D_{\tilde{S}'}$ on the surface $S'$ with the same configuration of the same kind of curves and the same coefficients. This is $\mathbb{Q}$-linearly equivalent to  $-K_{\tilde{S}'}$.
It is obvious that we can recover the birational morphisms $h$ and $g$, in such a way that the divisor  $D_{\tilde{S}'}$ plays the same role as $D_{\tilde{S}}$,  by tracking back the blow downs and blow ups along the way given in the table for the given singularity type.

Now we consider the case when the divisor $D_{\tilde{S}}$ in the table for the given singularity type contains a non-negative curve.
If we find a $\mathbb{Q}$-divisor $D_{\tilde{S}'}$ on the surface $S'$ with the same configuration of the same kind  of curves  and the same coefficients, then the method presented in the table works for the surface $S'$, as in the previous case.
To find such a $\mathbb{Q}$-divisor $D_{\tilde{S}'}$, we first notice from the table that the divisor $D_{\tilde{S}}$ contains at most
one non-negative curve. Let $F$ be the non-negative curve on $\tilde{S}$ that appears in $D_{\tilde{S}}$ with coefficient $a>0$. We have to show that such a non-negative curve always exists
on the surface $\tilde{S}'$. To do so,  put $D_{\tilde{S}}^0=D_{\tilde{S}}-aF$. We can  then find  a $\mathbb{Q}$-divisor $D_{\tilde{S}'}^0$ on the surface $S'$ with the same configuration of the same kind of curves and the same coefficients as $D_{\tilde{S}}^0$.
Next we find a composition $\psi$ of  $9-d$ blow downs  starting from $\tilde{S}$  to $\mathbb{P}^2$. Let $C_1, \ldots, C_{9-d}$ be the negative curves contracted by the birational morphism $\psi$.  We suppose that the first $r$ curves $C_1, \cdots, C_r$ (possibly $r=0$) intersect $F$ and the others do not intersect $F$.
We are then able to obtain the composition $\psi'$ of  the $9-d$ blow downs  starting from $\tilde{S}'$  to $\mathbb{P}^2$ by contracting the negative curves $C_1', \ldots, C_{9-d}'$ corresponding to the curves $C_1, \ldots, C_{9-d}$, respectively,  since the configurations of the negative curves on $\tilde{S}$  and $\tilde{S}'$ are the same.
Then we see the divisor $\psi(D_{\tilde{S}})$ on $\mathbb{P}^2$.  The curve $F$  is not contracted by $\psi$.  Now we see that finding
a $\mathbb{Q}$-divisor $D_{\tilde{S}'}$ on $\tilde{S}'$ is equivalent to finding an irreducible curve $F'$ of degree $\deg(\psi(F))$ on
$\mathbb{P}^2$ such that
\begin{itemize}
\item $\psi'(D_{\tilde{S}'}^0)+aF'$ and $\psi(D_{\tilde{S}})$ have the same configuration;
\item $F'$ contains  the points $\psi'(C_1'), \ldots, \psi'(C_r')$ but not the points $\psi'(C_{r+1}'), \ldots, \psi'(C_{9-d}')$.
\end{itemize}
It is straightforward to find such an irreducible curve on $\mathbb{P}^2$.

We can immediately find the negative curves for the morphisms $\psi$ from the configurations in the third column for the singularity types with $\textnormal{E}_6$ on del Pezzo surfaces of degree $1$. For the singularity types with $\textnormal{A}_4$ on del Pezzo surfaces of degree $1$
we keep it  in mind that  there is always  one $(-1)$-curve that meets the two $(-2)$-curves that are the ends of the chain of four $(-2)$-curves on $\tilde{S}$ (see \cite[Appendix]{Pa03}).
For the singularity type $\textnormal{A}_2$  on a del Pezzo surface of degree $2$, we provide more detail in Example~\ref{example:dp2-A2}. This also helps us understanding how to use the tables.

\begin{example}\label{example:dp2-A2}
We explain how to construct a cylinder on a del Pezzo surface  of degree $2$ with singularity type $\textnormal{A}_2$.

On the projective plane $\mathbb{P}^2$, take $D_{\mathbb{P}^2}=\frac{7}{4}L_1+\frac{5}{4}L_2$, where $L_1$ and $L_2$ are distinct two lines. As shown in the third column for $\textnormal{A}_2$ ($d=2$), we take ten blow ups following the depicted instruction. Let $h:\check{S}\to\mathbb{P}^2$ be the composition of these ten blow ups. As explained at the beginning of the section, $E_{\one}$
(resp. $E_{\two}, \ldots, E_{\ten}$)  is  the proper transform of the exceptional divisor of the first (resp. second, ... , tenth) blow up on the surface $\check{S}$.
We then obtain
\[K_{\check{S}}+D_{\check{S}}
=h^*\left(K_{\mathbb{P}^2}+D_{\mathbb{P}^2}\right)\sim_{\mathbb{Q}} 0,\]
where
\[\begin{split}D_{\check{S}}=&\left(\frac{3}{4}E_{\one}+\frac{1}{4}E_{\four}+\frac{1}{4}E_{\five}+\frac{1}{4}E_{\six}+\frac{1}{4}E_{\seven}+\frac{1}{4}E_{\eight}+\frac{1}{4}E_{\nine}+\frac{1}{4}E_{\ten}+\frac{5}{4}L_2\right)+\\ &\left(\frac{6}{4}E_{\two}+\frac{5}{4}E_{\three}+\frac{7}{4}L_1\right).\\
\end{split}\]
Here,  the proper transforms of $L_1$ and $L_2$ by $h$ are denoted by  the same notation.
The $\mathbb{Q}$-divisor  $D_{\check{S}}$ is obtained by the sum of two $\mathbb{Q}$-divisors in the second column of the table.
On the surface $\check{S}$, the curve  $L_2$ is a $(-5)$-curve,
the curve  $E_{\one}$ is a  $(-3)$-curve, the curves  $E_{\two}$,   $E_{\three}$ are $(-2)$-curves and the other eight curves in the second column of the table are $(-1)$-curves.

Starting from the $(-1)$-curve $L_1$, we can contract $E_{\two}$ and $E_{\three}$ in turn  to the smooth weak del Pezzo surface $\tilde{S}$  corresponding to a del Pezzo surface $S$ of degree $2$ with singularity type $\textnormal{A}_2$.
Denote the composition of these three blow downs by $g:\check{S}\to \tilde{S}$. Put
$$D_{\tilde{S}}=g\left(\frac{3}{4}E_{\one}+\frac{1}{4}E_{\four}+\frac{1}{4}E_{\five}+\frac{1}{4}E_{\six}+\frac{1}{4}E_{\seven}+\frac{1}{4}E_{\eight}+\frac{1}{4}E_{\nine}+\frac{1}{4}E_{\ten}+\frac{5}{4}L_2\right).$$
This is an effective anticanonical $\mathbb{Q}$-divisor on the surface $\tilde{S}$.

Note that the curves $g(E_{\one})$  and $g(L_2)$ are the only $(-2)$-curves on the surface $\tilde{S}$ and they intersect each other in  the form of $\textnormal{A}_2$. Contracting these two $(-2)$-curves, we obtain a birational morphism $f:\tilde{S}\to S$, where $S$ is a del Pezzo surface of degree $2$ with one singular point of type $\textnormal{A}_2$.
 Put
$$D_{S}=f\circ g\left(\frac{1}{4}E_{\four}+\frac{1}{4}E_{\five}+\frac{1}{4}E_{\six}+\frac{1}{4}E_{\seven}+\frac{1}{4}E_{\eight}+\frac{1}{4}E_{\nine}+\frac{1}{4}E_{\ten}\right).$$
This is an effective anticanonical $\mathbb{Q}$-divisor on the surface $S$ such that  $$S\setminus \mathrm{Supp}(D_S)\cong
\mathbb{P}^2\setminus \mathrm{Supp}(D_{\mathbb{P}^2})\cong \mathbb{A}^1\times \left(\mathbb{A}^1\setminus \{\mbox{one point}\}\right).$$

Now we consider an arbitrary del Pezzo surface $S'$  of degree $2$ with one singular point of type $\textnormal{A}_2$. Let $f':\tilde{S}'\to S'$ be the minimal resolution of the surface $S'$. The surface $\tilde{S}'$ is a smooth weak del Pezzo surface of degree $2$. Since it has the same configuration of negative curves as that of the weak del Pezzo surface $\tilde{S}$, we have the negative curves $E_{\one}'$, $E_{\two}'$, $E_{\three}'$, $E_{\five}', \ldots, E_{\ten}'$, $L_2'$ on the surface $\tilde{S}'$ corresponding to
$g(E_{\one})$, $g(E_{\two})$, $g(E_{\three})$, $g(E_{\five}), \ldots, g(E_{\ten})$, $g(L_2)$, respectively, on the surface $\tilde{S}$.  In order to construct a $(-K_{S'})$-polar cylinder on the surface $S'$, it is enough to show that we can obtain the same kind irreducible curve $E_{\four}'$ on the surface $\tilde{S}'$
as the $0$-curve $g(E_{\four})$ on the surface $\tilde{S}$.

Let $\psi' :\tilde{S}'\to\mathbb{F}_2$ be the birational morphism obtained by contracting the  six  $(-1)$-curves 
 $E_{\five}'$, $E_{\six}'$, $E_{\seven}'$, $E_{\eight}'$, $E_{\nine}'$, $E_{\ten}'$ to the Hirzeburch surface with $(-2)$-curve section. Instead of $\mathbb{P}^2$ we maps $\tilde{S}'$ to $\mathbb{F}_2$ because this gives simpler explanation. However, its principle is the same. 
 The image $\psi'(E_{\one}')$ is the negative section of $\mathbb{F}_2$. The image $\psi'(L_2')$ is irreducible and not contained in a fiber of $\mathbb{F}_2\to\mathbb{P}^1$. The curve $\psi'(L_2')$ intersects the section $\psi'(E_{\one}')$ at a single point.

 We have the fiber of $\mathbb{F}_2\to\mathbb{P}^1$  passing though  the intersection point
 of  $\psi'(L_2')$ and $\psi'(E_{\one}')$.
The proper transform of this fiber by $\psi'$ will play the role of $E_{\four}'$. To be precise, denote the proper transform of the fiber by $E_{\four}'$.  Then we put
$$D_{\tilde{S}'}=\frac{3}{4}E_{\one}'+\frac{1}{4}E_{\four}'+\frac{1}{4}E_{\five}'+\frac{1}{4}E_{\six}'+\frac{1}{4}E_{\seven}'+\frac{1}{4}E_{\eight}'+\frac{1}{4}E_{\nine}'+\frac{1}{4}E_{\ten}'+\frac{5}{4}L_2'.$$
This is an effective anticanonical $\mathbb{Q}$-divisor on the surface $\tilde{S}'$.  We put
$$D_{S'}=f'\left(\frac{1}{4}E_{\four}'+\frac{1}{4}E_{\five}'+\frac{1}{4}E_{\six}'+\frac{1}{4}E_{\seven}'+\frac{1}{4}E_{\eight}'+\frac{1}{4}E_{\nine}'+\frac{1}{4}E_{\ten}'\right).$$
This is an effective anticanonical $\mathbb{Q}$-divisor on the surface $S'$ and we have
$$S'\setminus \mathrm{Supp}(D_{S'})\cong
\tilde{S}'\setminus \mathrm{Supp}(D_{\tilde{S}'})\cong \mathbb{A}^1\times \left(\mathbb{A}^1\setminus \{\mbox{one point}\}\right).$$
Therefore, $S'$  has a $(-K_{S'})$-polar cylinder.
\end{example}

\begin{remark}
In fact, we have some freedom for the coefficients in the divisors $D_{\check{S}}$. We have fixed their coefficients simply to have better exposition in the table. For instance,  let us reconsider Example~\ref{example:dp2-A2}. We here consider
\[D_{\mathbb{P}^2}=\left(2-\epsilon\right)L_1+\left(1+\epsilon\right)L_2\]
instead of $\frac{7}{4}L_1+\frac{5}{4}L_2$. The proper transform of the divisor $D_{\mathbb{P}^2}$ by the birational morphism $h$ is
\[\begin{split}D_{\check{S}}=&\left(\left(1-\epsilon\right)E_{\one}+\left(1-3\epsilon\right)E_{\four}+\epsilon E_{\five}+\epsilon E_{\six}+\epsilon E_{\seven}+
\epsilon E_{\eight}+\epsilon E_{\nine}+\epsilon E_{\ten}+\left(1+\epsilon\right)L_2\right)+\\ &\left(\left(2-2\epsilon\right)E_{\two}+\left(2-3\epsilon\right)E_{\three}+
\left(2-\epsilon\right)L_1\right).\\
\end{split}\]
For the divisor $D_{\check{S}}$ to be effective and to contain the exceptional divisors of the birational morphisms $h$ and $g$, it is enough to take a rational number $\epsilon$ such that   $0< \epsilon <\frac{1}{3}$. In Example~\ref{example:dp2-A2}, we have simply chosen $\epsilon=\frac{1}{4}$. In almost all the other singularity types of the table, we may manipulate   the coefficients in the divisors $D_{\check{S}}$ in the same way.
\end{remark}
\bigskip


\begin{center}

\begin{longtable}{|l|c|c|}
\caption{Degree 1}\label{table:deg1}\\
\hline

\begin{minipage}[m]{.15\linewidth}\setlength{\unitlength}{.25mm}
\begin{center}
Singularity

Type
\end{center}
\end{minipage}
&\begin{minipage}[m]{.40\linewidth}\setlength{\unitlength}{.25mm}

\begin{center}
\medskip
Tiger/

Divisor contracted (if any)
\medskip
\end{center}
\end{minipage} & Construction \\

\hline
\begin{minipage}[m]{.15\linewidth}\setlength{\unitlength}{.25mm}
 $\textnormal{E}_8$

  {\tiny $\textnormal{E}_7$}

 \end{minipage}
&

\begin{minipage}[m]{.40\linewidth}\setlength{\unitlength}{.25mm}

$2E_{\one}+4E_{\two}+6E_{\three}
+5E_{\four}+4E_{\five}+3E_{\six}+2E_{\seven}+
E_{\eight}^{\bone}+3L$

\end{minipage}
&
\begin{minipage}[m]{.31\linewidth}\setlength{\unitlength}{.25mm}
\begin{center}
\begin{picture}(400,50)(-10,15)

\put(130,42){\eight} \put(137,42){\line(1,0){30}}

\thicklines

\put(-10,25){\one}
\put(10,45){\two}
\put(28,25){\three}
\put(52,45){\four}
\put(70,25){\five}
\put(92,45){\six} \put(110,25){\seven}
\put(33,51){\mbox{\tiny $L$}}

\put(0,30){\line(2,1){30}}

 \put(20,45){\line(2,-1){30}}
 \put(40,30){\line(2,1){30}}
\put(60,45){\line(2,-1){30}}
 \put(80,30){\line(2,1){30}}
\put(100,45){\line(2,-1){30}}
\put(120,30){\line(2,1){30}}
 \qbezier(40,50)(58,36)(68,20)

\put(168,42){\circle*{3}}
\end{picture}

\end{center}

\end{minipage}
\\
\hline

\end{longtable}
\end{center}


\begin{center}
\begin{longtable}{|l|c|c|}
\caption{Degree 2}\label{table:deg2}\\
\hline

\begin{minipage}[m]{.15\linewidth}\setlength{\unitlength}{.25mm}
\begin{center}
Singularity

Type
\end{center}
\end{minipage}
&\begin{minipage}[m]{.40\linewidth}\setlength{\unitlength}{.25mm}

\begin{center}
\medskip
Tiger/

Divisor contracted (if any)
\medskip
\end{center}
\end{minipage} & Construction \\

\hline $\textnormal{D}_4$ &
\begin{minipage}[m]{.40\linewidth}\setlength{\unitlength}{.25mm}

$2E_{\one}+\frac{4}{3}E_{\two}+\frac{1}{3}E_{\three}^{\bone}+\frac{1}{3}E_{\four}^{\bone}
+\frac{1}{3}E_{\five}^{\bone}+\frac{1}{3}E_{\six}^{\bone}+\frac{1}{3}E_{\seven}^{\bone}
+\frac{4}{3}L_{1}+\frac{4}{3}L_2+\frac{1}{3}L_3^{\bone}$

\end{minipage}
&
\begin{minipage}[m]{.31\linewidth}\setlength{\unitlength}{.25mm}
\begin{center}
\begin{picture}(200,85)(-10,-7)

\put(75,61){\mbox{\tiny $L_3$}}

\put(9,63){\mbox{\tiny $L_2$}}

\qbezier(73,60)(63,50)(63,40)
\put(0,50){\mbox{\tiny $L_1$}}

\put(-5,12){\line(1,0){30}}
\put(-5,7){\line(1,0){30}}

\put(30,20){\line(1,0){30}}
\put(30,25){\line(1,0){30}}

\put(88,32){\line(1,0){30}}

\put(-9,25){\mbox{\tiny $\one$}}

\put(79,28){\mbox{\tiny$\three$}}

\put(-13,2){\mbox{\tiny$\five$}}
 \put(31,60){\mbox{\tiny$\two$}}
\put(-13,11){\mbox{\tiny $\four$}}
 \put(21,15){\mbox{\tiny$\seven$}}
\put(21,23){\mbox{\tiny $\six$}}

\thicklines

\put(0,30){\line(2,1){60}}
\put(40,60){\line(2,-1){60}}
\qbezier(20,60)(40,30)(40,10)
\qbezier(7,45)(7,25)(7,0)

\put(25,12){\circle*{3}}\put(25,7){\circle*{3}}
\put(60,20){\circle*{3}}\put(60,25){\circle*{3}}
\put(119,32){\circle*{3}} \put(63,40){\circle*{3}}

\end{picture}
\end{center}
\end{minipage}
\\
\hline
\end{longtable}
\end{center}

\bigskip

\textbf{Acknowledgements.} The authors would like to express their sincere appreciation to the referee for
the invaluable comments. The referee's comments enable the authors to improve their results as well as their
exposition. 
The first author was supported within the framework of a subsidy granted to the HSE
by the Government of the Russian Federation for the implementation of the Global Competitiveness Program.
The second author has been supported by IBS-R003-D1, Institute for Basic Science in Korea and the third author has been supported by NRF-2014R1A1A2056432, the National Research Foundation in Korea.


\end{document}